\documentclass[12pt]{article}
\usepackage{amsmath, amsthm}
\usepackage{amscd}
\usepackage{amssymb}
\usepackage{array}
\usepackage{color}

\newtheorem{thm}{Theorem}[section]
\newtheorem{theorem}[thm]{Theorem}

\newtheorem{corollary}[thm]{Corollary}
\newtheorem{lemma}[thm]{Lemma}
\newtheorem{proposition}[thm]{Proposition}

\theoremstyle{definition}
\newtheorem{definition}[thm]{Definition}
\newtheorem{example}[thm]{Example}

\newtheorem{remark}[thm]{Remark}

\begin{document}

\newcommand{\comment}[1]
{{\color{blue}\rule[-0.5ex]{2pt}{2.5ex}}
\marginpar{\small\begin{flushleft}\color{blue}#1\end{flushleft}}}

\newcommand{\id}{\relax{\rm 1\kern-.28em 1}}
\newcommand{\R}{\mathbb{R}}
\newcommand{\C}{\mathbb{C}}
\newcommand{\Z}{\mathbb{Z}}
\newcommand{\Q}{\mathbb{Q}}
\newcommand{\bD}{\mathbb{D}}
\newcommand{\bG}{\mathbb{G}}
\newcommand{\bP}{\mathbb{P}}
\newcommand{\g}{\mathfrak{G}}
\newcommand{\e}{\epsilon}
\newcommand{\cA}{\mathcal{A}}
\newcommand{\cB}{\mathcal{B}}
\newcommand{\cC}{\mathcal{C}}
\newcommand{\cD}{\mathcal{D}}
\newcommand{\cI}{\mathcal{I}}
\newcommand{\cL}{\mathcal{L}}
\newcommand{\cO}{\mathcal{O}}
\newcommand{\cG}{\mathcal{G}}
\newcommand{\cJ}{\mathcal{J}}
\newcommand{\cF}{\mathcal{F}}
\newcommand{\cP}{\mathcal{P}}
\newcommand{\cU}{\mathcal{U}}
\newcommand{\ep}{\mathcal{E}}
\newcommand{\E}{\mathcal{E}}
\newcommand{\cH}{\mathcal{O}}
\newcommand{\cV}{\mathcal{V}}
\newcommand{\cPO}{\mathcal{PO}}
\newcommand{\cHol}{\mathrm{Hol}}
\newcommand{\cp}{\mathcal{P}}

\newcommand{\rGL}{\mathrm{GL}}
\newcommand{\rSU}{\mathrm{SU}}
\newcommand{\rSL}{\mathrm{SL}}
\newcommand{\rPSL}{\mathrm{PSL}}
\newcommand{\rSO}{\mathrm{SO}}
\newcommand{\rOSp}{\mathrm{OSp}}
\newcommand{\rSpin}{\mathrm{Spin}}
\newcommand{\rsl}{\mathrm{sl}}
\newcommand{\rM}{\mathrm{M}}
\newcommand{\rdiag}{\mathrm{diag}}
\newcommand{\rP}{\mathrm{P}}
\newcommand{\rdeg}{\mathrm{deg}}
\newcommand{\pt}{\mathrm{pt}}
\newcommand{\red}{\mathrm{red}}

\newcommand{\bm}{\mathbf{m}}

\newcommand{\M}{\mathrm{M}}
\newcommand{\End}{\mathrm{End}}
\newcommand{\Hom}{\mathrm{Hom}}
\newcommand{\diag}{\mathrm{diag}}
\newcommand{\rspan}{\mathrm{span}}
\newcommand{\rank}{\mathrm{rank}}
\newcommand{\Gr}{\mathrm{Gr}}
\newcommand{\ber}{\mathrm{Ber}}

\newcommand{\str}{\mathrm{str}}
\newcommand{\Sym}{\mathrm{Sym}}
\newcommand{\tr}{\mathrm{tr}}
\newcommand{\defi}{\mathrm{def}}
\newcommand{\Ber}{\mathrm{Ber}}
\newcommand{\spec}{\mathrm{Spec}}
\newcommand{\sschemes}{\mathrm{(sschemes)}}
\newcommand{\sschemeaff}{\mathrm{ {( {sschemes}_{\mathrm{aff}} )} }}
\newcommand{\rings}{\mathrm{(rings)}}
\newcommand{\Top}{\mathrm{Top}}
\newcommand{\sarf}{ \mathrm{ {( {salg}_{rf} )} }}
\newcommand{\arf}{\mathrm{ {( {alg}_{rf} )} }}
\newcommand{\odd}{\mathrm{odd}}
\newcommand{\alg}{\mathrm{(alg)}}
\newcommand{\sa}{\mathrm{(salg)}}
\newcommand{\sets}{\mathrm{(sets)}}
\newcommand{\smflds}{\mathrm{(smflds)}}
\newcommand{\mflds}{\mathrm{(mflds)}}
\newcommand{\SA}{\mathrm{(salg)}}
\newcommand{\salg}{\mathrm{(salg)}}
\newcommand{\varaff}{ \mathrm{ {( {var}_{\mathrm{aff}} )} } }
\newcommand{\svaraff}{\mathrm{ {( {svar}_{\mathrm{aff}} )}  }}
\newcommand{\ad}{\mathrm{ad}}
\newcommand{\Ad}{\mathrm{Ad}}
\newcommand{\pol}{\mathrm{Pol}}
\newcommand{\Lie}{\mathrm{Lie}}
\newcommand{\Proj}{\mathrm{Proj}}
\newcommand{\rGr}{\mathrm{Gr}}
\newcommand{\rFl}{\mathrm{Fl}}
\newcommand{\rPol}{\mathrm{Pol}}
\newcommand{\rdef}{\mathrm{def}}

\newcommand{\uspec}{\underline{\mathrm{Spec}}}
\newcommand{\uproj}{\mathrm{\underline{Proj}}}

\newcommand{\sym}{\cong}

\newcommand{\al}{\alpha}
\newcommand{\lam}{\lambda}
\newcommand{\de}{\delta}
\newcommand{\ttau}{\tilde \tau}
\newcommand{\D}{\Delta}
\newcommand{\s}{\sigma}
\newcommand{\lra}{\longrightarrow}
\newcommand{\ga}{\gamma}
\newcommand{\ra}{\rightarrow}

\newcommand{\NOTE}{\bigskip\hrule\medskip}

\medskip

\centerline{\LARGE \bf On SUSY curves}

\vskip 1cm

\centerline{ R. Fioresi, S. D. Kwok}

\medskip

\centerline{\it Dipartimento di Matematica, Universit\`{a} di
Bologna }
 \centerline{\it Piazza di Porta S. Donato, 5. 40126 Bologna. Italy.}
\centerline{{\footnotesize e-mail: rita.fioresi@unibo.it,
stephendiwen.kwok@unibo.it}}

\section{Introduction} \label{intro-sec}

In this note we give a summary of some elementary
results in the theory of super Riemann surfaces (SUSY curves),
which are mostly known, but are not readily available in the
literature. Our main source is Manin, who has provided
with a terse introduction to this subject in \cite{ma1}. More
recently Freund and Rabin have given important results on the uniformization
(see \cite{rabin}) and Witten has written an 
account of the state of the art of this subject, 
from the physical point of view, in \cite{witten}.

\medskip
\noindent
The paper is organized as follows.

\smallskip

In Sections \ref{prelim-sec} and \ref{pigeom-sec}, 
we are going to recall briefly the
main definitions of supergeometry and study in detail the examples
of super projective space and $\Pi$-projective line, which are very
important in the theory of SUSY curves.

\smallskip

In Section \ref{susyc-sec} we discuss some general facts on SUSY curves,
including the {\sl theta characteristic}, while in  
Section \ref{genus1-sec} we prove some characterization results
concerning genus zero and genus one SUSY curves.

\section{Preliminaries} \label{prelim-sec}

We are going to briefly recall some basic definitions of
analytic supergeometry. For more details see \cite{ma1}, \cite{ma2},
\cite{vsv2}, \cite{ccf}, \cite{dm} and the classical references
\cite{Berezin}, \cite{BerLeites}.

\medskip
\noindent
Let our ground field be $\C$.

\medskip

A {\it superspace} $S=(|S|, \cO_S)$ is a topological space $|S|$
endowed with a sheaf of superalgebras $\cO_S$ such that the stalk at
a point  $x\in |S|$, denoted by $\cO_{S,x}$, is a local
superalgebra.

\smallskip

A {\it morphism} $\phi:S \lra T$ of superspaces is given by
$\phi=(|\phi|, \phi^*)$, where $|\phi|: |S| \lra |T|$ is a map of
topological spaces and $\phi^*:\cO_T \lra \phi_*\cO_S$ is such
that $\phi_x^*(\bm_{|\phi|(x)})=\bm_x$ where $\bm_{|\phi|(x)}$ and
$\bm_{x}$ are the maximal ideals in the stalks $\cO_{T,|\phi|(x)}$
and $\cO_{S,x}$ respectively. 

\medskip

The superspace $\C^{p|q}$ is the topological space $\C^p$ endowed with
the following sheaf of superalgebras. For any open subset $U \subset
\C^p$
$$
\cO_{\C^{p|q}}(U)=\cHol_{\C^p}(U)\otimes \wedge(\xi_1, \dots, \xi_q),
$$
where 
$\cHol_{\C^p}$ denotes the complex analytic sheaf on $\C^p$ and
$\wedge(\xi_1, \dots, \xi_q)$ is the exterior algebra in the variables
$\xi_1, \dots, \xi_q$.

\medskip

A  {\it supermanifold} of dimension $p|q$ is a superspace $M=(|M|,
\cO_M)$ which is locally isomorphic to $\C^{p|q}$, as superspaces.
A {\it morphism} of supermanifolds is simply a morphism of
superspaces.

\medskip

We now look at an important example of supermanifold, namely the
{\sl projective superspace}. 

\smallskip

Let $\bP^m=\C^{m+1}\setminus\{0\} \big/ \sim$ 
be the ordinary complex projective space of dimension $m$
with homogeneous coordinates $z_0, \dots, z_m$; $[z_0, \dots, z_m]$
denotes as usual an equivalence class in $\bP^m$.
Let $\{U_i\}_{i=1,\dots, m}$ be the affine cover $U_i=
\{[z_0, \dots, z_m] \, | \, z_i \neq 0\}$,
$U_i\cong \C^m$. On each $U_i$ we take the global ordinary coordinates
$u_0^i, \dots, \hat u_i^i, \dots u_m^i$, $u_k:=z_k/z_i$
($ \hat u_i^i$ means we are omitting the variable $u_i^i$
from the list). We now want to 
define the sheaf of superalgebras $\cO_{U_i}$ on the topological
space $U_i$: 
$$
\cO_{U_i}(V)=\cHol_{U_i}(V) \otimes \wedge(\xi_1^i, \dots \xi_n^i), 
\qquad V \, \hbox{open in} \, U_i 
$$
where $\cHol_{U_i}$ is the sheaf of holomorphic functions on $U_i$
and $\xi_1^i, \dots \xi_n^i$ are odd variables. 

\smallskip

As one can readily check $\cU_i=(U_i,\cO_{U_i})$ is a supermanifold, isomorphic
to $\C^{m|n}$. We now define the morphisms
$\phi_{ij}:\cU_i \cap \cU_j \mapsto  \cU_i \cap \cU_j$,
where the domain is thought as an open submanifold of $\cU_i$,
while the codomain as an open submanifold of $\cU_j$. By the Chart's
Theorem the $\phi_{ij}$'s are determined by
the ordinary morphisms, together with the choice of $m$ even and 
$n$ odd sections in
$\cO_{U_i}(U_i \cap U_j)$. We write:
\begin{equation}
\phi_{ij}:\,(u_0^i, \dots, \hat u_i^i, \dots u_m^i, \xi_1^i, \dots , \xi_n^i) 
\, \mapsto \, 
\left(\frac{u_1^i}{u_j^i}, \dots, \frac{1}{u_j^i}, \dots , 
\frac{u_m^i}{u_j^i},  \frac{\xi_1^i}{u_j^i}, \dots , 
\frac{\xi_n^i}{u_j^i} \right) \label{proj-trans-maps}
\end{equation}
where on the right hand side the $1/u_j^i$ appears in the $i^{th}$ position and
the $j^{th}$ position is omitted. 
As customary in the literature, the formula (\ref{proj-trans-maps}) is just a synthetic
way to express the pullbacks:
$$
\phi_{ij}^*(u^j_k) \, = \, \frac{u_k^i}{u_j^i}, \quad 0 \leq k \neq j \leq m, \qquad
\phi_{ij}^*(\xi^j_l)\, = \, \frac{\xi_l^i}{u_j^i}, \quad 0 \leq l \leq n,
$$
One can easily check that the $\phi_{ij}$'s satisfy the compatibility
conditions:
$$
\phi_{ij}\phi_{ki}=\phi_{jk}, \qquad \hbox{on} \quad \cU_i \cap \cU_j \cap \cU_k
$$
hence they allow us to define uniquely a sheaf, denoted with 
$\cO_{\bP^{m|n}}$, hence a supermanifold
structure on the topological space $\bP^m$. The supermanifold
$(\bP^m, \cO_{\bP^{m|n}})$ is called the \textit{projective space of dimension
$m|n$}.

\smallskip

One can replicate the same construction and obtain more generally 
a supermanifold structure for the topological space:
$\bP(V):= V \setminus \{0\}$, where $V$ in any complex super vector space.

\medskip

We now introduce the functor of points approach to supergeometry.

\medskip

The \textit{functor of points} of a supermanifold $X$ is the functor (denoted
with the same letter)
$X:\smflds^o \lra \sets$, $X(T)=\Hom(T,X)$, $X(f)\phi=f \circ \phi$.
The functor of points characterizes completely the supermanifold $X$:
in fact, two supermanifolds are isomorphic if and only if their functor
of points are isomorphic. This is one of the statements of Yoneda's
Lemma, for more details see \cite{ccf} ch. 3. 

\medskip

The functor of points approach allows us to retrieve some of the
geometric intuition. For example, let us consider the functor 
$P:\smflds^o \lra \sets$ associating to each supermanifold $T$ the
locally free subsheaves of $\cO_T^{m+1|n}$ of rank $1|0$, where
$\cO_T^{m+1|n}:=\C^{m+1|n}\otimes \cO_T$. $P$ is defined in an obvious way
on the morphism: any morphism of supermanifolds $\phi:T \lra S$
defines a corresponding morphism of the structural sheaves
$\phi^*:\cO_S \lra  \phi_*\cO_T$, so that also $P(\phi)$
is defined.

\smallskip

The next proposition allows us
to identify the functor $P$ with the functor of points of
the super projective space $\bP^{m|n}$.

\begin{proposition}\label{fopt-proj-prop}
There is a one-to-one correspondence between the two sets:
$$
P(T) \longleftrightarrow \bP^{m|n}(T), \qquad T \in \smflds
$$
which is functorial in $T$. In other words $P$ is the functor
of points of $\bP^{m|n}(T)$.
\end{proposition}

\begin{proof} 
We briefly sketch the proof, leaving to the reader the routine
checks. Let us start with an element in $P(T)$, that is a
locally free sheaf $\cF_T \subset \cO_T^{m+1|n}$ of rank $1|0$. 
We want to associate to $\cF_T$
a $T$-point of $\bP^{m|n}$ that is a morphism $T \lra\bP^{m|n}$. 
First cover $T$ with $V_i$ so that $\cF_T|_{V_i}$ is free. Hence:
$$
\cF_T(V_i)=\rspan\left\{(t_0,  \dots, t_m,\theta_1, \dots, \theta_n)\right\}
\, \subset \, \cO_T^{m+1|n}(V_i)
$$
where we assume that the section $t_i \in \cO_T(V_i)$ is invertible
without loss of generality, since the 
rank of  $\cF_T$ is $1|0$ (this assumption may require to change the
cover). 
Hence:
$$
\cF_T(V_i)=\rspan\left\{(t_0/t_i,  \dots, 1, \dots, t_m/t_i,
\theta_1/t_i, \dots, \theta_n/t_i)\right\}
$$
Any other basis $(t'_1, \dotsc, t'_{m+1},$ $ \theta'_1, \dotsc, \theta'_n)$ 
of $\mathcal{F}_T(V_i)$ is a multiple of 
$(t_0,  \dots, t_m$, $\theta_1, \dots, \theta_n)$ 
by an invertible section on $V_i$, hence we have:
$$
t'_j/t'_i = t_j/t_i\qquad 
\theta'_k/t'_i = \theta_k/t_i
$$
\noindent Thus the functions $t_j/t_i, \theta_k/t_i$, which a priori 
are only defined on open subsets where $\mathcal{F}_T$ is a free module, 
are actually defined on the whole of the open set where $t_i$ is 
invertible, being independent of the choice of basis for $\mathcal{F}_T(V_i)$.

\smallskip\noindent
We have then immediately a morphism of supermanifolds
$f_i:\cV_i \lra \cU_i \subset \bP^{m|n}$: 
\begin{equation}
f_i^*(u_1^i)=t_0/t_i, \dots, f_i^*(u_m^i)= t_m/t_i, \quad
 f_i^*(\xi_1^i)=\theta_1/t_i ,\dots,  f_i^*(\xi_n^i)=\theta_n/t_i
\end{equation}
where $\cV_i=(V_i, \cO_T|_{V_i})$ and $\cU_i=(U_i, \cO_{\bP^{m|n}}|_{U_i})$.
It is immediate to check that the $f_i$'s agree on $\cV_i  \cap \cV_j$,
so they glue to give
a morphism $f: T \lra  \bP^{m|n}$.

\medskip

As for the vice versa, consider $f:T \lra  \bP^{m|n}$ and define 
$V_i = |f|^{-1}(U_i)$. The morphism $f|_{V_i}$ by the Chart's Theorem
corresponds to the choice of $m$ even and $n$ odd sections in $\cO_T(V_i)$:
$v_1, \dots, v_m$, $\eta_1, \dots, \eta_n$. 
We can then define immediately the free sheaves
$\cF_{V_i}\subset \cO_T|_{V_i}^{m+1|n}$ of rank $1|0$ on each of the $V_i$ as
$$
\cF_{V_i}(V)\, := \, \rspan \{(v_1|_{V}  ,\dots,1, \dots v_m|_V,
\eta_1|_V ,\dots, \eta_n|_V) \}
$$
($1$ in the $i^{th}$ position).
As one can readily check the $\cF_{V_i}$ glue to give a locally
free subsheaf of $\cO_T^{m+1|n}$. 
\end{proof}

We now want to define the {\sl $\Pi$-projective line} which represents 
in some sense a generalization of the super projective space of 
dimension $1|1$ that 
we defined previously. 

\medskip

Let $\bP^1=\C^2\setminus \{0\}/ \sim$ be the ordinary complex projective 
line with homogeneous coordinates $z_0$, $z_1$. 
Define, as we did before, the following supermanifold structure
on each $U_i$ belonging to the open cover $\{U_0, U_1\}$ of $\bP^1$:
$\cO_{U_i}(V)$ $=$ $\cHol_{U_i}(V) \otimes \wedge(\xi)$, 
$V$ open in $U_i$, $i=1,2$, so that 
$\cU_i=(U_i, \cO_{U_i})$ is a supermanifold isomorphic to $\C^{1|1}$.
At this point, instead of the change of chart $\phi_{12}$, 
we define the following
transition map (there is only one such):
$$
\begin{array}{cccc}
\psi_{12}: & \cU_0 \cap \cU_1 & \lra & \cU_0 \cap \cU_1 \\ \\
& (u, \xi) & \mapsto & \left(\frac{1}{u}, -\frac{\xi}{u^2}\right)  
\end{array}
$$

As one can readily check, this defines a supermanifold structure on 
the topological space $\bP^1$ and
we call this supermanifold the \textit{$\Pi$-projective line}
$\bP^{1|1}_\Pi=(\bP^1, \cO_{\bP^1_{\Pi}})$.

\medskip

In the next section we will 
characterize its functor of points.

\section{The $\Pi$-projective line}
\label{pigeom-sec}

In this section we want to take advantage of the functor of points
approach in order to give a more geometric point of view on the
$\Pi$-projective line and to understand in which sense it is a
generalization of the super projective line, whose functor of points was
described in the previous section.
Let us start with an overview of the
ordinary geometric construction of the projective line.

\smallskip
The topological space $\bP^1$ consists
of the 1-dimensional subspaces of $\C^2$, that is
$\bP^1=\C^2 \setminus \{0\}/\sim$, where $(z_0,z_1) \sim (z_0',z_1')$
if and only if $(z_0,z_1) =\lambda (z_0',z_1')$, $\lambda \in \C^\times$.
In other words, the equivalence class $[z_0,z_1] \in \bP^1$ consists
of all the points in $\C^2$ which are in the orbit of $(z_0,z_1)$
under the action of $\C^\times$ by left (or right) multiplication. 

\smallskip

Now we go to the functor of points of $\bP^{1|1}$.
A $T$-point of $\bP^{1|1}$ locally is a $1|0$-submodule 
of $\cO_T^{1|1}(V)$ ($V$ is a suitably
chosen open in $T$). So it is locally an equivalence class 
$[z_0,z_1,\eta_0, \eta_1]$ where we identify two
quadruples $(z_0,z_1,\eta_0, \eta_1) \sim (z_0',z_1',\eta_0', \eta_1')$ 
if and only if $z_i=\lambda z_i'$
and $\eta_i=\lambda \eta_i'$, $i=0,1$, 
$\lambda \in \cO_T(V)^\times$. 
In other words, exactly
as we did before for the case of $\bP^1$,  we identify those elements
in $\C^{2|1}(T)$ that belong to the same orbit of 
the multiplicative group of the complex
field $\bG_m^{1|0}(T) \cong \C^\times(T)$.\footnote{All of our arguments 
here take place for an open cover of $T$ in which a $T$ point corresponds 
to a free sheaf  and not just a locally free one. For simplicity of exposition we
omit to mention the cover and the necessary gluing to make all of our argument 
stand.} 
It makes then perfect sense to generalize this construction and look at the 
equivalence classes
with respect to the action of the multiplicative supergroup $\bG_m^{1|1}$,
which is the supergroup with underlying topological space $\C^\times$, with
one global odd coordinate and with group law
(in the functor of points notation):
$$
(a, \al) \cdot (a', \al') = (a  a' + \al \al', a \al'+  \al a').
$$
$\bG_m^{1|1}$ is naturally embedded into $\rGL(1|1)$, the complex general 
linear supergroup
via the morphism (in the functor of points notation):
$$
\begin{array}{ccc}
\bG_m^{1|1}(T) & \lra & \rGL(1|1)(T) \\ \\
(a, \al) & \mapsto & \begin{pmatrix}a & \al \\ \al & a \end{pmatrix}
\end{array}
$$
This is precisely the point of view we are taking in constructing the $\Pi$-projective
line: we identify $T$-points in $\C^{2|2}$ which lie in the same $\bG_m^{1|1}$ orbit, but
instead of looking simply at rank $1|1$ submodules of $\C^{2|2}(T)$ 
we look at a more elaborate structure, which is
matching very naturally the $\bG_m^{1|1}$ action on $\C^{2|2}$. This structure is embodied
by the condition of $\phi$-invariance for a suitable odd endomorphism $\phi$
of $\C^{2|2}$, that we shall presently see. For more details see Appendix 
\ref{D-app}. 

\medskip\noindent 
Consider now 
the supermanifold $\C^{2|2}$, and the odd endomorphism $\phi$ on $\C^{2|2}$ 
given in terms of the standard homogeneous basis $\{e_0, e_1 | \ep_0, \ep_1\}$ by:

\begin{align*}
\Bigg( \begin{array}{cc|cc}
0 & 0 & 1& 0\\
0 & 0 & 0 & 1\\
\hline
1 & 0 & 0 & 0\\
0 & 1 & 0 & 0
\end{array} 
\Bigg)
\end{align*}

\noindent We note that $\phi^2 = 1$. 

In analogy with the projective superspace, we now consider 
the functor $P_\Pi: \smflds^o \lra \sets$,
where 
$$
\begin{array}{rl}
P_\Pi(T) &:= 
\{ \text{rank $1|1$ locally free, $\phi$-invariant subsheaves of } \cO^{2|2}_T\}
\end{array}
$$ 
Here the action of $\phi$ is extended to 
$\cO^{2|2}_T = \mathbb{C}^{2|2} \otimes_{\C} \cO_T$ 
by acting on the first factor.\\

\begin{lemma} \label{basis}
Let $\cF_T \in P_\Pi(T)$. Then there exist an open cover $\{V_i\}$ of $T$,
where $\cF_T(V_i)$ is free and  
a basis $e$, $\ep$ of $\cF_T(V_i)$ such that $\phi(e) = \ep, \phi(\ep) = e$. 
\end{lemma}

\begin{proof}
Since $\cF_T$ is locally free, there exist an open cover $\{V_i\}$ of $T$,
where $\cF_T(V_i)$ is free with basis, say, $e'$, $\ep'$. Let $\Psi$ 
be the matrix of 
$\phi|_{V_i}$ in this basis. Since $\phi^2 = 1$, we have $\Psi^2 = 1$, which implies 
that $\Psi$ has the form:
\begin{align*}
\Psi := 
\bigg( \begin{array}{c|c}
\alpha & a \\
\hline
a^{-1} & -\alpha
\end{array} 
\bigg)\\
\end{align*}
with $a \in \mathcal{O}^*_T(V_i)_0$, $\alpha \in \mathcal{O}_T(V_i)_1$. 
Let $P \in \rGL(1|1)(\mathcal{O}_T(V_i))$
be the matrix:
\begin{align*}
P := 
\bigg( \begin{array}{c|c}
a^{-1} & 0 \\
\hline
-a^{-1}\alpha & 1
\end{array} 
\bigg)\\
\end{align*}
$P$ is invertible because $a$ is, and one calculates 
that $P\Psi P^{-1} = \Phi$, so $P$ gives the desired change of basis.
\end{proof}

\begin{proposition}\label{fopts-pi}
There is a one-to-one correspondence between the two sets:
$$
P_\Pi(T) \lra \bP^{1|1}_\Pi(T), \qquad T \in \smflds
$$
which is functorial in $T$. In other words $P_\Pi$ is the functor 
of points of $\bP^{1|1}_\Pi$.
\end{proposition}

\begin{proof} 
We briefly sketch the proof, leaving to the reader the 
routine checks. Let us consider 
a locally free sheaf $\cF_T \subset \cO_T^{2|2}$ of rank $1|1$ in $P_\Pi(T)$, 
invariant under $\phi$. We want to associate to each such $\cF_T$
a $T$-point of $\bP^{1|1}_\Pi(T)$, that is, a morphism $f:T \lra\bP^{1|1}_\Pi$. 

\smallskip\noindent
First we cover $T$ with $V_i$, so that $\cF_T|_{V_i}$ is free. 
By Lemma \ref{basis} 
there exists a basis $e$, $\ep$ of $\cF_T(V_i)$ such that 
$\phi(e) = \ep$, $\phi(\ep) = e$. 

\smallskip\noindent
Representing $e$, $\ep$ using the basis 
$\{e_0, \ep_0, e_1, \ep_1\}$ of $\C^{2|2}$, we have:
$$
\cF_T(V_i)=\rspan\left\{e\,=\,(s_0,\sigma_0,s_1,\sigma_1), 
\, \ep\,=\,(\sigma_0, s_0, \sigma_1, s_1) \right\}
$$
for some sections $s_j, \sigma_j$ in $\cO_T(V_i)$. Since the rank of $\cF_T$ is 
$1|1$, either $s_0$ or $s_1$ must
be invertible. Let us call $V_0$ the union of the $V_i$ for which
$s_0$ is invertible and $V_1$ the union of the $V_i$ for which $s_1$ is
invertible. 

\smallskip\noindent
Hence, we can make a change of basis of $\cF_T(V_i)$ by right
multiplying the column vectors representing $e$ and $\ep$ 
in the given basis, by a suitable element 
$g_i \in \rGL(1|1)(\cO_T(V_i))$
obtaining: 
$$
\begin{array}{c}
\begin{array}{rl}
\cF_T(V_0)=\rspan &\{(1,0,s_1s_0^{-1}-\sigma_1\sigma_0s_0^{-2}, \,
\sigma_1s_0^{-1}-s_1 \sigma_0 s_0^{-2}),\\ \\
&(0,1, \sigma_1s_0^{-1}-s_1\sigma_0s_0^{-2}, \,
s_1s_0^{-1}-\sigma_1\sigma_0s_0^{-2})\},\\ \\ \\
\cF_T(V_1)=\rspan 
&\{(s_0s_1^{-1}-\sigma_0\sigma_1 s_1^{-2},\, \sigma_0 s_1^{-1} -s_0 
\sigma_1 s_1^{-2}, 1,0),\\ \\
&( \sigma_0 s_1^{-1} -s_0 \sigma_1 s_1^{-2}, \,
s_0s_1^{-1}-\sigma_0\sigma_1 s_1^{-2}, 0,1)\}, \\
\end{array} \\ \\
g_0=
\begin{pmatrix}s_0^{-1} & -\sigma_0s_0^{-2} \\ 
-\sigma_0s_0^{-2}   & s_0^{-1} \end{pmatrix},
\qquad g_1=
\begin{pmatrix}s_1^{-1} & -\sigma_1s_1^{-2} \\
-\sigma_1s_1^{-2}   & s_1^{-1} \end{pmatrix} 
\end{array}
$$

Suppose now $\{e', \ep'\} := 
\{(s'_0,\sigma'_0,s'_1,\sigma'_1), (\sigma'_0, s'_0, \sigma'_1, s'_1)\}$ 
is another basis of $\mathcal{F}_T(V_i)$ such that 
$\phi(e') = \ep', \phi(\ep') = e'$. The sections in
$\cO_T(V_i)$ we have obtained, namely:
$$
\begin{array}{cc}
v_0\,=\,s_1s_0^{-1}-\sigma_1\sigma_0s_0^{-2}, & 
\nu_0\,=\,\sigma_1s_0^{-1}-s_1\sigma_0s_0^{-2} \\ \\
v_1\,=\,s_0s_1^{-1}-\sigma_0\sigma_1 s_1^{-2}, & 
\nu_1 \,=\,\sigma_0 s_1^{-1} -s_0 
\sigma_1 s_1^{-2} 
\end{array}
$$
are independent of the choice of such a basis. This can be easily
seen with an argument very similar to the one in Prop. \ref{fopt-proj-prop}.

\smallskip\noindent
Hence we have well-defined morphisms of supermanifolds 
$f_i:\cV_i \lra \cU_i \subset \bP^{1|1}_\Pi$: 
$$
\begin{array}{cc}
f_0^*(u_0)=v_0, &  f_0^*(\xi_0)=\nu_0 \\ \\
f_0^*(u_1)=v_1, &  f_0^*(\xi_1)=\nu_1  
\end{array}
$$
where $\cV_i=(V_i, \cO_T|_{V_i})$ and $\cU_i=(U_i, \cO_{\bP^{1|1}_\Pi}|_{U_i})$,
while $(u_i,\xi_i)$ are global coordinates on $\cU_i \cong \C^{1|1}$.
A small calculation shows that the $f_i$'s agree on $\cV_0  \cap \cV_1$,
in fact as one can readily check: 
$$
(1,0,v_0,\nu_0) \sim \left(\frac{1}{v_0},\frac{-\nu_0}{v_0^{2}},1,0\right)
$$
and similarly for $(v_1,\nu_1,1,0)$, 
which corresponds to the transition map for $\bP^{1|1}_\Pi$
we defined in Sec. \ref{prelim-sec}.
So the $f_i$'s glue to give a morphism $f: T \lra  \bP^{1|1}_\Pi$.

\medskip

For the converse, consider $f:T \lra  \bP^{1|1}_\Pi$ and define 
$V_i = |f|^{-1}(U_i)$. We can define immediately the sheaves
$\cF_{V_i} \subset \cO^{2|2}_{V_i}$ on each of the $V_i$ as we did in the proof 
of Prop.
\ref{fopt-proj-prop}: 

\begin{align*}
\cF_{V_0} &= \text{span }\{(1, 0, t_0, \tau_0), (0, 1, \tau_0,t_0)\}\\
\cF_{V_1} &= \text{span }\{(t_1, \tau_1, 1, 0), (\tau_1, t_1, 0, 1)\}\\
\end{align*}

\noindent where $t_i = f^*(u_i), \tau_i = f^*(\xi_i)$. The $\cF_{V_i}$ so defined are free of rank $1|1$ (by inspection there are no nontrivial relations between the generators), and $\phi$-invariant by construction. Finally, one checks that the relations $t_1 = t_0^{-1}, \tau_1 = -t_0^{-2}\tau_0$ imply that the $\cF_{V_i}$ glue on $V_0 \cap V_1$ to give a locally free rank $1|1$ subsheaf of $\cO_T^{2|2}$.
\end{proof}

\section{Super Riemann surfaces} \label{susyc-sec}

In this section we give the definition of super Riemann surface
and we examine some elementary, yet important properties.

\smallskip\noindent
Much of the material we discuss in this section is contained, though not so explicitly,
in \cite{ma2}.

\begin{definition}
A $1|1$-\textit{super Riemann surface} is a pair $(X, \cD)$, where $X$ is a 
$1|1$-dimensional complex supermanifold, and $\cD$ is a locally direct 
(and consequently locally free, by the super Nakayama's lemma) rank $0|1$ 
subsheaf of the tangent sheaf $TX$ such that:
$$
\begin{array}{cccc}
\cD{\otimes} \cD & \lra &TX/\cD& \\ \\
Y \otimes Z &\mapsto &[Y,Z] &(mod \, \cD)
\end{array}
$$
is an isomorphism of sheaves. Here $[ \, , \, ]$ denotes the super 
Lie bracket of vector fields. The distinguished subsheaf $\cD$ is called 
a \textit{SUSY-1 structure} on $X$, and $1|1$-super Riemann surfaces are 
thus alternatively referred to as \textit{SUSY-1 curves}. 
We shall refer to SUSY 1-structures simply as SUSY structures. 

\smallskip\noindent
We say that $X$ has genus $g$ if the underlying topological
space $|X|$ has genus $g$.
\end{definition}

\medskip

\begin{definition} 
Let $(X, \mathcal{D})$, $(X', \mathcal{D}')$ be SUSY-1 curves, 
and $F: X \to X'$ a biholomorphic map of supermanifolds. 
$F$ is a {\it isomorphism of SUSY curves}, or simply a 
{\it SUSY isomorphism}, if $(dF)_p(\cD_p) = \cD'_{|F|(p)}$ for all reduced points 
$p \in |X|$. Here $(dF)_p$ denotes the differential of $F$ at $p$, 
$\cD_p \subset T_pX$ (resp. $\cD'_q$) the stalk of the subsheaf 
$\mathcal{D}$ (resp. $\mathcal{D}'$) at $p$ (resp. $q$).
\end{definition}

\begin{example} \label{examplesusy}
Let us consider the supermanifold $\C^{1|1}$, with
global coordinates $z$, $\zeta$ together with the
odd vector field:
$$
V=\partial_\zeta+\zeta \partial_z
$$\\
If $\cD=\rspan\{V\}$, $\cD$ is a SUSY structure on $\C^{1|1}$ since
$V$, $V^2$ span $T\C^{1|1}$. As we will see, this is the {\sl unique}
(up to SUSY isomorphism) SUSY structure on $\C^{1|1}$.
\end{example}


We now want to relate the SUSY structures 
on a supermanifold and the canonical bundle of the reduced underlying 
manifold. It is important to remember that for a supermanifold $X$ of 
dimension $1|1$, $\cO_{X,0} = \cO_{X,\red}$; that is, the even part 
of its structural sheaf coincides with its reduced part. 
This is of course not true for a generic supermanifold.

\smallskip\noindent
We start by showing that any SUSY structure can be locally put into a {\sl canonical form}.

\begin{lemma} \label{localsusy-lemma}
Let $(X, \cD)$ be a SUSY-1 curve, $p$ a topological point in $X_\red$. 
Then there exists an open set $U$ containing $p$ and a coordinate system 
$W = (w, \eta)$ for $U$ such that $\mathcal{D}|_U =$ span $\{\partial_\eta + \eta \partial_w\}$.
\end{lemma}

\begin{proof}
Since $\mathcal{D}$ is locally free, there exists a neighborhood $U$ of $p$ on which $\mathcal{D} =$ span $\{D\}$, where $D$ is some odd vector field; by shrinking $U$, we may assume it is also a coordinate domain with coordinates $(z, \zeta)$. Since $X$ has only one odd coordinate, we have $D = f(z) \partial_\zeta+ g(z)\zeta \partial_z$ 
for some holomorphic even functions $f$, $g$.
So:
$$
D^2=[D,D]/2=g(z) \partial_z +g f' \zeta  \partial_\zeta.
$$
Since $D$, $D^2$ form a free local basis for the $\cO_X$-module $TX$, we have
$$
a_{11}D + a_{12}D^2 = \partial_z, \qquad 
a_{21}D + a_{22}D^2 = \partial_\zeta
$$
If we substitute the expression for $D$ and $D^2$ we obtain:
$$
g(a_{11} \zeta + a_{12}) = 1, \qquad  a_{21} f = 1 - a_{22} gf'\zeta
$$
from which we conclude that both $f$ and $g$ must be units.

\smallskip\noindent
We now show that we can find a new coordinate system (possibly shrinking $U$) so that $D$ can be
put in the desired form. We will assume such a coordinate system exists, then
determine a formula for it and this formula will give us the existence.
Let $w = w(z)$, $\eta = h(z) \zeta$ be the new coordinate system, where 
$w$ and $h$ are holomorphic functions. By the
chain rule, we have:
$$
\partial_z=w'(z)\partial_w+h'(z) \zeta \partial_\eta, \qquad
\partial_\zeta=h(z)\partial_\eta
$$
We now set $D=\partial_\eta+\eta \partial_w$ and substituting we have:
$$
D=\partial_\eta+\eta \partial_w=fh\partial_\eta  + gh^{-1}\eta w'(z)
\partial_w
$$
which holds if and only if the system of equations:
\begin{align*}
fh=1,\qquad
gw'=h
\end{align*}
has a solution for $w, h$. By shrinking our original coordinate domain, we may assume it is simply connected. Then since $f$ and $g$ are units, the system has a solution, by standard facts from complex analysis. We leave to the reader the easy check that $(w,\eta)$ is indeed a coordinate system.
\end{proof}

\begin{definition}Let $(X, \mathcal{D})$ be a SUSY-1 curve, $U$ an open set. Any coordinate system $(w, \eta)$ on $U$ having the property of Lemma \ref{localsusy-lemma} is said to be {\it compatible} with the SUSY structure $\cD$. Any open set $U$ that admits a $\mathcal{D}$-compatible coordinate system is said to be {\it compatible} with $\mathcal{D}$.
\end{definition}

\begin{definition}
Let $X_\red$ be an ordinary Riemann surface, $K_{X_\red}$ its canonical bundle. 
A \textit{theta characteristic} is a pair $(\cL, \alpha)$, where $\cL$ is a holomorphic line bundle on $X_\red$, and $\alpha$ a holomorphic isomorphism of line bundles $\al: \cL \otimes \cL \lra K_{X_\red}$. An {\it isomorphism} of theta characteristics $(L, \alpha)$, $(L', \alpha')$ is an isomorphism $\phi: L \to L'$ of line bundles such that $\alpha' \circ \phi^{\otimes 2} = \alpha$. 

\smallskip\noindent
Some authors also call a theta characteristic of $X_\red$ a \textit{square root} of the canonical bundle $K_{X_\red}$.
\end{definition}

\begin{definition}
A {\it super Riemann pair}, or {\it SUSY pair} for short, is a pair $(X_r, (\mathcal{L}, \alpha))$ where $X_r$ is an ordinary Riemann surface, and $(\mathcal{L}, \alpha)$ is a theta characteristic on $X_r$. An {\it isomorphism} of SUSY pairs $F: (X_r, (\mathcal{L}, \alpha)) \to (X'_r, (\mathcal{L}', \alpha'))$ is a pair $(f, \phi)$ where $f: X_r \to X'_r$ is a biholomorphism of ordinary Riemann surfaces, and $\phi: \mathcal{L}' \to f_*(\mathcal{L})$ is an isomorphism of theta characteristics on $X'_r$. 
\end{definition}

For the sake of brevity, we will occasionally omit writing the isomorphism $\alpha$ in describing a super Riemann pair. The following theorem shows that the data of super Riemann surface and of super Riemann pair are completely equivalent.

\begin{theorem} \label{catequivalence-thm}
Let $(X, \mathcal{D})$ be a SUSY-1 curve. Then $(X_{\red}, \cO_{X, 1})$ is a SUSY pair, where $\cO_{X,1}$ is regarded as an $\cO_{X, 0}$-line bundle. 
Furthermore, if $F := (f, f^\#): (X, \mathcal{D}) \to (X', \mathcal{D}')$ is a SUSY-isomorphism, then $(f, f^\#|_{\mathcal{O}_{X,1}}): (X_{\red}, \mathcal{O}_{X,1}) \to (X'_{\red}, \mathcal{O}_{X', 1})$ is an isomorphism of SUSY pairs. 

Conversely, suppose $(X_{r}, (\cL, \alpha))$ is a SUSY pair. Then there exists a structure of SUSY-1 curve $(X_\mathcal{L}, \mathcal{D}_{\cL})$ on $X_r$, such that the SUSY pair associated to $(X_{\cL}, \mathcal{D})$ equals $(X_{r}, (\cL, \alpha))$. Any isomorphism of SUSY pairs $(X_{r}, (\cL, \alpha)) \to (X'_{r}, (\cL', \alpha'))$ induces a SUSY-isomorphism $X_{\mathcal{L}} \to X_{\mathcal{L}'}$.
\end{theorem}

\begin{proof} 
First we show that if $X$ is a $1|1$-complex supermanifold with a SUSY structure, then the $\cO_{X,0}$-line-bundle $\cO_{X,1}$ is a square root of the canonical bundle $K_{X_{\red}}$. 

\smallskip\noindent
By Lemma \ref{localsusy-lemma}, $X$ has an open cover by compatible coordinate charts. If $(z,\zeta)$ and $(w,\eta)$ are two 
such coordinate charts, then 
$$
D_z=\partial_\zeta+\zeta\partial_z, \qquad
D_w=\partial_\eta+\eta\partial_w 
$$ 
with $D_z=h(z)D_w$, $h(z) \neq 0$. If $w=f(z)$ and $\eta=g(z)\zeta$
a small calculation implies that $f'(z)=g^2$, that is, $\mathcal{O}_{X,1}^{\otimes 2}$
and $K_{X_\red}$ have the same transition functions for this covering, hence there is an isomorphism $\mathcal{O}_{X, 1}^{\otimes 2} \to K_{X_\red}$. If $F := (f, f^\#): (X, \mathcal{D}) \to (X', \mathcal{D}')$ is a SUSY-isomorphism of SUSY-1 curves with underlying Riemann surface $X_{\red}$, then one checks that $f^\#|_{\mathcal{O}_{X',1}}: \mathcal{O}_{X', 1} \to f_*(\mathcal{O}_{X, 1})$ is an isomorphism of line bundles. Covering $X$ with an atlas of compatible coordinate charts, transferring this atlas to a compatible atlas on $X'$ by $F$, and comparing the transition functions for $f_*(\mathcal{O}_{X,1})$ and $\mathcal{O}_{X',1}$ in this atlas as above, we obtain the desired isomorphism of theta characteristics.
\smallskip

Conversely, if we have a theta characteristic 
$\alpha: \mathcal{L}^{\otimes 2} \to K_{X_{\red}}$, we define a sheaf of supercommutative 
rings $\mathcal{O}_{X_{\mathcal{L}}}$ on $|X|$, the topological space underlying $X_\red$, 
by setting:
\begin{align*}
\mathcal{O}_{X_{\mathcal{L}}} = \mathcal{O}_{X_\red} \oplus \mathcal{L}
\end{align*}
with multiplication $(f, s) \cdot (g, t) = (fg, ft + gs)$. One checks that 
$\mathcal{O}_{X_\mathcal{L}}$ so defined is a sheaf of local supercommutative rings, using the standard fact 
that a supercommutative ring $A$ is local if and only if its even part $A_0$ is local. 
By taking a local basis $\chi$ of $\mathcal{L}$ in a trivialization, and sending 
$(f, g\chi)$ to $f + g \eta$, we see that $\mathcal{O}_{X_{\mathcal{L}}}$ so defined is 
locally isomorphic to $\mathcal{O}_{X_\red} \otimes \Lambda[\eta]$ and 
hence $(X_\red, \mathcal{O}_{X_{\mathcal{L}}})$ is a supermanifold.

\smallskip\noindent
The SUSY structure is defined as follows. Let $z$ be a coordinate for $X_\red$ on 
an open set $U$. By shrinking $U$ we may assume $\mathcal{O}_\cL(U)$ is free. 
Then there is some basis $\zeta$ of $\cO_{\mathcal{L}}(U)$ such that 
$\alpha(\zeta \otimes \zeta) = dz$; then $z, \zeta$ so defined are coordinates 
for $X_\cL$ on $U$. We set the SUSY structure on $U$ to be that spanned by 
$D_Z := \partial_\zeta + \zeta \partial_z$.

We will show the local SUSY structure thus defined is independent of our choices 
and hence is global on $X$. Suppose $w$ is another coordinate on $U$, and $\eta$ 
a basis of $\cO_{\mathcal{L}}(U)$ such that $\alpha(\eta \otimes \eta) = dw$. 
Then $w = f(z), \eta = g(z) \zeta$, with $f'$ a unit in $U$. Since $dw = f'(z) \, dz$, 
we have:
\begin{align*}
\alpha(\eta \otimes \eta) &= g^2 \alpha(\zeta \otimes \zeta)\\
&=f'(z) dz
\end{align*}
from which it follows that $g^2 = f'$; in particular, $g$ is also a unit. 
Then by the chain rule, $\partial_\zeta + \zeta \partial_z = g(\partial_\eta + \eta \partial_w)$,
hence $D_Z$ and $D_W$ span the same SUSY structure on $U$.

\smallskip\noindent
Now suppose $(X'_r, \mathcal{L}')$ is another SUSY pair, isomorphic to $(X, \mathcal{L})$ by $(f, \phi)$. Then $\phi$ will induce an isomorphism of analytic supermanifolds
$\psi: X_{\cL} \lra X_{\cL'}$, since $f: X_{\cL,\red} \to X_{\cL', \red}$ is an isomorphism, and 
$f_*(\cO_{X_\cL,1}) \cong \cO_{X_{\cL'},1}$ via the isomorphism
$\phi$ of theta characteristics.
Now we check we have a SUSY isomorphism. This may be done locally: given a point $p \in X_\red$ 
one chooses coordinates $(z, \zeta)$ and $(z', \zeta')$ around $p$ that are compatible 
with the SUSY-1 structures on $X_\cL$ and $X_{\cL'}$, so that 
$D_Z := \partial_\zeta + \zeta \partial_z$ 
(resp. $D_{Z'} := \partial_{\zeta'} + \zeta' \partial_{z'}$) locally generate the SUSY 
structures. In these coordinates the reader may check readily that 
$d{\psi}(D_Z|_p) = D_{Z'}|_p$. 
\end{proof}

Theorem \ref{catequivalence-thm} has the following important
immediate consequence.

\begin{corollary} \label{cateq-cor}
A $1$-dimensional complex manifold $X_\red$ carries a SUSY structure 
if and only if $X$ admits a theta characteristic.
\end{corollary}

\begin{remark} 
One can prove, via a direct argument using cocycles, that any 
compact Riemann surface $S$ admits a theta characteristic, using 
the fact that the Chern class is $c_1(K_S) = 2 - 2g$ (i.e. it
is divisible by $2$ hence $K_S$ admits a square root). 
Hence SUSY-1 curves exist in abundance: any compact Riemann surface 
admits at least one structure of SUSY-1 curve.
%
\end{remark}

\smallskip\noindent
Theorem \ref{catequivalence-thm} has also the following important consequences:

\begin{proposition}\label{susyc1|1-prop}
Up to SUSY-isomorphism, there is a unique SUSY-1 structure 
on $\mathbb{C}^{1|1}$, namely, that defined by the odd vector field:

\begin{align*}
V = \partial_\zeta + \zeta \partial_z
\end{align*}

\noindent where $(z, \zeta)$ are the standard linear coordinates 
on $\mathbb{C}^{1|1}$. 
\end{proposition}

\begin{proof}
The reduced manifold of $\mathbb{C}^{1|1}$ is $\mathbb{C}$. It is well known that all holomorphic line bundles on $\mathbb{C}$ are trivial. 
This implies there exists only one theta characteristic for $\mathbb{C}$ up to isomorphism, namely $(\mathcal{O}_{\mathbb{C}}, \id^{\otimes 2})$. The essential point in verifying this uniqueness is that any automorphism of trivial line bundles on $\mathbb{C}$ is completely determined by an invertible entire function on $\mathbb{C}$, and such a function always has an invertible entire square root. Hence by
Theorem \ref{catequivalence-thm}, there is only one SUSY-1 structure on $\mathbb{C}$ up to isomorphism. For the last statement of the theorem, see Example \ref{examplesusy}.
\end{proof}

The next example shows that Lemma \ref{localsusy-lemma} is a purely local result.

\begin{example}
Consider the vector field:
$$
Z = \partial_\zeta + e^z \zeta \partial_z
$$
$Z$ is an odd vector field on $\mathbb{C}^{1|1}$ 
defining a SUSY structure on $\mathbb{C}^{1|1}$. 
The previous proposition implies that the SUSY structure 
defined by $Z$ is isomorphic to that defined by 
$V=\partial_\zeta + \zeta \partial_z$. 
However, it does {\it not} imply that there exists a global coordinate 
system $(w, \eta)$ for $\mathbb{C}^{1|1}$ in which $Z$ takes the 
form $\partial_\eta + \eta \partial_w$. In fact 
suppose such a global coordinate system 
$w = f(z), \eta = g(z) \zeta$ existed. Then:

\begin{align*}
\bigg(
\begin{array}{c|c}
f' & 0\\
\hline
g' \zeta & g
\end{array}
\bigg)
\bigg(
\begin{array}{c}
e^z \zeta\\
\hline
1
\end{array}
\bigg) = \bigg(
\begin{array}{c}
\eta\\
\hline
1
\end{array}
\bigg)
\end{align*}

\noindent from which we conclude that $g = 1, f'= e^{-z}$. 
Hence $f = -e^{-z} + c$, but since $f$ is not one-to-one, 
this contradicts the assumption that $(w, \eta)$ is a coordinate system 
on all of $\mathbb{C}^{1|1}$. This shows that Lemma \ref{localsusy-lemma}
cannot be globalized even in the simple case of $\mathbb{C}^{1|1}$, even though $\mathbb{C}^{1|1}$ has a unique SUSY structure up to SUSY isomorphism. 
\end{example}

\section{Super Riemann surfaces of genus zero and one} 
\label{genus1-sec}

In this section we want to provide some classification results
on SUSY curves of genus zero and one. 
The next proposition provides a complete classification of compact 
super Riemann surfaces of genus zero and shows the existence 
of a genus zero $1|1$ compact complex supermanifold,
namely the $\Pi$-projective line, that does 
not admit a SUSY structure.

\begin{proposition}
\begin{enumerate}
\item $\bP^{1|1}$ admits a unique SUSY structure, up to SUSY-isomorphism. More
generally, if $X$ is a supermanifold of dimension $1|1$ of
genus zero, then $X$ admits
a SUSY structure if and only if $X$ is isomorphic to $\bP^{1|1}$.\\
\item $\mathbb{P}^1_\Pi$ admits no SUSY structure.
\end{enumerate}
\end{proposition}

\begin{proof}
To prove $(1)$, recall the well-known classification of line bundles on $\bP^1$: $Pic(\bP^1)$ is a free abelian group of rank one, generated by the isomorphism class of the hyperplane bundle $\mathcal{O}(1)$, and $K_{\bP^1} \cong \mathcal{O}(-2)$.

Hence, up to isomorphism, there is a unique theta characteristic on $\bP^1$, namely ($\mathcal{O}(-1), \psi)$ where $\psi$ is any fixed isomorphism of line bundles $\mathcal{O}(-1)^{\otimes 2} \to \mathcal{O}(-2)$. Similar to Prop \ref{susyc1|1-prop}, the proof of this uniqueness reduces to the problem of lifting a given global automorphism of $\mathcal{O}(-1)^{\otimes 2} \cong \mathcal{O}(-2)$ to an automorphism of $\mathcal{O}(-1)$. This requires the fact that $End(L) = L^* \otimes L = \mathcal{O}$ for any line bundle $L$, and that $H^0(\mathbb{P}^1, \mathcal{O}) = \mathbb{C}$. In particular, any global automorphism of $\mathcal{O}(-1)^{\otimes 2}$ is given by multiplication by an invertible scalar, which has an invertible square root in $\mathbb{C}$; this is the desired automorphism of $\mathcal{O}(-1)$. 

\smallskip\noindent
Considering now the statement $(2)$, by Theorem \ref{catequivalence-thm}, if 
$X = \bP^{1|1}_\Pi$ admitted a SUSY-1 structure, we would have $\mathcal{O}_{X, 1} \cong \mathcal{O}(-1)$. Using the coordinates from Section 2, we see that $\mathcal{O}_{X, 1} \cong \mathcal{O}(-2)$. This is a contradiction.
\end{proof}

We now turn to the study of genus one SUSY curves. 

\medskip\noindent
In ordinary geometry a compact Riemann surface $X$ of genus one, that is
an elliptic curve, is obtained by quotienting $\C$ by a lattice $L \cong \Z^2$. 
It is easily seen that any such lattice $L$ is equivalent, under scalar 
multiplication, to a lattice of the form $L_0 := \rspan\{1, \tau\}$, 
where $\tau$ lies in the upper half plane. 
Two lattices $L_0=\rspan\{1, \tau\}$ and
$L_0'=\rspan\{1, \tau'\}$, are equivalent, i.e., yield 
isomorphic elliptic curves, if and only if $\tau$ and $\tau'$ lie in
the same orbit of the group $\Gamma=\rPSL_2(\Z)$, where
the action is via linear fractional transformations:
$$
\tau \mapsto \frac{a\tau + b}{c\tau+d}
$$
A fundamental domain for this action is:
$$
D \, = \, \{ \tau \in \C \, | \, Im(\tau)>0, \, |Re(\tau)| \leq 1/2, \, 
|\tau| \geq 1 \} 
$$

We now want to generalize this picture to the super setting. Our
main reference will be \cite{rabin}.

\smallskip

We start by observing the ordinary action of 
$\Z^2 \cong L_0=\langle A_0, B_0 \rangle$ on $\C$ is given explicitly by:
$$
A_0: z \mapsto z+1, \qquad B_0: z \mapsto z+\tau
$$

\smallskip\noindent
In \cite{rabin} Freund and Rabin take a similar point of view in constructing a super Riemann surface: they define even super elliptic curves as quotients  
of $\mathbb{C}^{1|1}$ by $\mathbb{Z}^2 = \langle A, B \rangle$, acting by: 
\begin{align*}
A:& (z, \zeta) \mapsto (z + 1, \pm \zeta)\\
B:& (z, \zeta) \mapsto (z + \tau, \pm \zeta),
\end{align*} 

\smallskip

In this section, we will justify their choice of these particular actions by showing they are the only reasonable generalizations of the classical actions of $\mathbb{Z}^2$ on $\mathbb{C}$. 

In Sec. \ref{susyc-sec} we proved that on $\C^{1|1}$ there exists, up to isomorphism, only one SUSY structure, corresponding to the vector 
field $V=\partial_\zeta + \zeta \partial_z$. We now want to characterize all possible SUSY automorphisms preserving this SUSY structure.

\smallskip
\noindent
We start with some lemmas.

\begin{lemma}\label{ker-lemma}
Let $X$ be a $1|1$ complex supermanifold and $\omega, \omega'$ 
be holomorphic $1|0$ differential forms on $X$ such that 
$ker(\omega), ker(\omega')$ are $0|1$ distributions. 
Then $ker(\omega) = ker(\omega')$ if and only if $\omega' = t \omega$ 
for some invertible even function $t(z)$.
\end{lemma}

\begin{proof}

The $\Leftarrow$ implication is clear. To prove the $\Rightarrow$ implication,
 we can reduce to a local calculation. Suppose now that 
$ker(\omega) = ker(\omega')$. Given any point $p$ in $X$, fix an 
open neighborhood $U \ni p$ where $TX|_U$ is free. As $\cD$ is locally a direct
summand, it is locally free of rank $0|1$, by the super Nakayama's lemma (see \cite{vsv2})
and $\cD |_U$ has a local complement $\mathcal{E} \subset TX|_U$ 
(shrinking $U$, if needed). 

\smallskip\noindent
Let us use the notation $\cO_U=\cO_X|_U$ and $\cO(U)=\cO_U(U)$.
As $\mathcal{E}$ is also a direct summand 
of $TX|_U$, it is also a free $\mathcal{O}_U$-module (again possibly 
shrinking $U$) hence must be of rank $1|0$. Hence we have a local splitting 
$TX|_U = \cD \oplus \mathcal{E}$ of free $\mathcal{O}_U$-modules. 
Let $Z$ be a basis for $\mathcal{D}|_U$, $W$ a basis for $\mathcal{E}$; 
then $W, Z$ form a basis for $TX|_U$.

\smallskip\noindent
$\omega|_U: \mathcal{O}_{TX}(U) \to \mathcal{O}_U(U)$ induces an even linear 
functional $\omega_p: T_pX \to \mathbb{C}$ on the tangent space at $p$, and the
 splitting $TX|_U = \cD \oplus \mathcal{E}$ induces a corresponding splitting 
$T_pX = D_p \oplus E_p$ of super vector spaces, with $dim(D_p) = 0|1$, $dim(E_p)
 = 1|0$, such that $ker(\omega_p) = D_p$ and span$(W_p) = E_p$. By linear 
algebra, $\omega_p|_{E_p}$ is an isomorphism; in particular, $\omega_p(W_p)$ is 
a basis for $\C$ = $\cO_p/\mathfrak{M}_p$. The super Nakayama's lemma then 
implies that $\omega(W)$ generates $\mathcal{O}_U$ as $\mathcal{O}_U$-module 
(again shrinking $U$ if necessary), which is true if and only if 
$\omega(W)$ is a unit; the same is true for $\omega(W')$.
 
\smallskip\noindent
We now show that the ratio $\omega'(W)/\omega(W) \in \mathcal{O}^*_{U,0}$ is 
independent of the local complement $\mathcal{E}$ and the choice of $W$, so 
that it defines an invertible even function $t$ on all of $X$. Suppose 
$\mathcal{E}'$ is another local complement to $\mathcal{D}$ on $U$, and $W'$ 
a local basis for $\mathcal{E}'$. We have as before that $Z, W'$ form a basis 
of $TX|_U$. Then $\omega(W'), \omega'(W')$ are invertible in $\mathcal{O}_U$ 
by the above argument, and $W' = bW + \beta Z$ with 
$b \in \mathcal{O}^*_{U, 0}$, $\beta \in \mathcal{O}_{U, 1}$.

\begin{align*}
\frac {\omega(W')} {\omega'(W')} &= \frac {b\omega(W) + \beta \omega(Z)} 
{b\omega'(W) + \beta \omega'(Z)}\\
&= \frac {\omega(W)} {\omega'(W)}
\end{align*}

\noindent Note here that we have used the hypothesis 
$ker(\omega) = ker(\omega')$ to conclude $\omega(Z) = \omega'(Z) = 0$.

Finally,
we verify that $\omega'= t \omega$; this can again be done locally 
since $t$ is now known to be globally defined. The argument is left to the reader.
\end{proof}

The odd vector field $V=\partial_\zeta + \zeta \partial_z$ defining
our SUSY structure $\cD$ is dual to the 
differential form $\omega := dz - \zeta \, d\zeta$.
As one can readily check $\mathcal{D} =$ span$\{V\} = ker(\omega)$.

\begin{lemma}\label{form-lemma}
An automorphism $F: \mathbb{C}^{1|1} \to \mathbb{C}^{1|1}$ is a SUSY automorphism if and only if 
$F^*(\omega) = t(z)\omega$ for some invertible even function $t(z)$. 
\end{lemma}

\begin{proof}
Unraveling the definitions, one sees that $F$ preserves the 
SUSY structure if and only if 
$$ker(F^*(\omega))_p = ker(\omega)_p$$
for each $p \in X$. We claim that the latter is true if and only if 
$ker(F^*(\omega)) = ker(\omega)$. 
One implication is clear. 
Conversely, suppose that $ker(F^*(\omega))_p = ker(\omega)_p$ for each 
$p \in X$. By a standard argument using the super Nakayama's Lemma, 
$ker(F^*(\omega)) = ker(\omega)$ in a neighborhood of 
$p$ for any point $p$, hence $ker(\omega) = ker(F^*(\omega))$. 
The result then follows by Lemma \ref{ker-lemma}.
\end{proof}

We are now ready for the result characterizing all of the 
SUSY automorphisms of $\C^{1|1}$. 

\begin{proposition} 
Let $(z, \zeta)$ be the standard linear coordinates on $\C^{1|1}$, 
and let $\C^{1|1}$ have the natural SUSY-1 structure defined by the vector field 
$V=\partial_\zeta + \zeta \partial_z$. The SUSY automorphisms of $\C^{1|1}$ 
are precisely the endomorphisms $F$ of $\mathbb{C}^{1|1}$ such that: 
$$
F(z,\zeta)=(az+b, \pm\sqrt{a} \zeta)
$$
\noindent where $a \in \mathbb{C}^*, b \in \mathbb{C}$, and $\sqrt{a}$ 
denotes either of the two square roots of $a$.
\end{proposition}

\begin{proof} 
Let $F$ be such an automorphism, and $z, \zeta$ the standard coordinates 
on $\C^{1|1}$. Then by the Chart Theorem, $F(z, \zeta) = (f(z), g(z) \zeta)$ 
for some entire functions $f, g$ of $z$. Similarly, $F^{-1}(z, \zeta) = 
(h(z), k(z) \zeta))$ for some entire functions $h, k$.
Since $F $ and $F^{-1}$ are inverses, $f$ is a biholomorphic 
automorphism of $\mathbb{C}^{1|0}$, hence linear by standard facts from complex 
analysis: $f(z) = az + b$ for some $a, b \in \mathbb{C}$, $a \neq 0$; 
the same is true for $h$.

So by the Lemma \ref{form-lemma}, 
$F$ preserves the SUSY-1 structure on $\C^{1|1}$ 
if and only if $F^*(\omega) = t(z) \omega$. We calculate: 

\begin{align*}
F^*(dz - \zeta \, d \zeta) &= df - F^*(\zeta) \, d(g \zeta)\\
&= f' \, dz - g^2 \zeta \, d \zeta.
\end{align*}\

\noindent Equating this with $t(z) \omega$, we see $t = f' = g^2$. 
Thus $g^2 = a$, so in particular $g$ is constant. Hence:

\begin{align*}
F(z, \zeta) = (az + b, c \zeta)
\end{align*}\

where $c^2 = a$, $a \in \mathbb{C} \backslash \{0\}$, $b \in \mathbb{C}$. 
Conversely, one checks that any morphism $\C^{1|1} \to \C^{1|1}$ of the above 
form is an automorphism, and that it preserves the SUSY structure.
\end{proof}

From our previous proposition, we conclude immediately that the only 
actions of $\mathbb{Z}^2$ on $\mathbb{C}^{1|1}$ that restrict to 
the usual action on the reduced space $\mathbb{C}$ are of the form:
\begin{align*}
A:& (z, \zeta) \mapsto (z + 1, \pm \zeta)\\
B:& (z, \zeta) \mapsto (z + \tau, \pm \zeta),
\end{align*}
since the actions of $A$ and $B$ must be by automorphisms of the form:
\begin{align*}
(z, \zeta) \mapsto (az + b, \pm \sqrt{a} \zeta)
\end{align*}
and in this case, $a$ must be taken to be $1$. This justifies the
choice made in \cite{rabin}.

\begin{remark} 
Using Theorem \ref{catequivalence-thm}, we see that the SUSY structures on $X_{\red}$ correspond one-to-one to isomorphism classes of theta characteristics on $X_{\red}$. It is well-known from the theory of elliptic curves over $\mathbb{C}$ that an elliptic curve $X_{\red}$ has four distinct theta characteristics, up to isomorphism. Regarding $X_{\red}$ as an algebraic group, these theta characteristics correspond to the elements of the subgroup of order $2$ in $X_{\red}$.

As noted in \cite{Levin}, one can define the {\it parity} of a theta characteristic $\cL$ as $dim \,H^0(X_{\red}, \mathcal{L}) \, (mod \; 2)$. This is a fundamental invariant of the theta characteristic (cf. \cite{Atiyah} where the parity is shown to be stable under holomorphic deformation). The isomorphism class of the trivial theta characteristic $\mathcal{O}_{X_{\red}}$ is distinguished from the other three by its parity: it has odd parity, the others have even parity. The odd case is therefore fundamentally different from the perspective of supergeometry, and is best studied in the context of {\it families} of super Riemann surfaces; families of odd super elliptic curves are considered in, for instance, \cite{rabin}, \cite{Levin}, \cite{witten}.
\end{remark}

\medskip\noindent
In \cite{rabin}, Rabin and Freund also describe a projective embedding
of the SUSY curve defined by $\C^{1|1}/\langle \, A,\, B \, \rangle$
using the classical Weierstrass function $\wp$ and the function $\wp_1$ defined
as $\wp_1^2=\wp-e_1$ (as usual $e_1=\wp(\omega_i)$ with $\omega_1=1/2$,
$\omega_2=\tau/2$ and $\omega_3=(1+\tau)/2$).
If $U_0$, $U_1$, $U_2$ is the open cover of $\bP^{2|3}(\C)$ described
in Sec. \ref{prelim-sec}, on $U_2$ the embedding is defined as:
$$
\begin{array}{ccc}
\C^{1|1}/\langle \, A,\, B \, \rangle & \lra & U_2 \subset  \bP^{2|3}(\C) \\ \\
(z,\zeta) & \mapsto & [\wp(z), \wp'(z), 1, \wp_1(z)\zeta, 
 \wp_1'(z)\zeta,  \wp_1(z)\wp(z)\zeta] 
\end{array}
$$
In \cite{rabin}, they describe also the equations of the ideal
in $U_2$ corresponding to the SUSY curve in this embedding:
$$
\begin{array}{c}
y^2= 4x^3-a_1x^2-a_2, \qquad 
2(x-e_1)\eta_2=y\eta_1 \\ \\
y\eta_2=2(x-e_2)(x-e_3)\eta_1 \qquad
\eta_3=x\eta_1
\end{array}
$$
where $(x,y,\eta_1, \eta_2, \eta_3)$ are the global coordinates
on $U_2\cong \C^{2|3}$. One can readily compute the homogeneous
ideal in the ring $\C[x_0,x_1,x_2, \xi_1, \xi_2,\xi_3]$ associated
with the given projective embedding. It is generated by the equations:
$$
\begin{array}{c}
x_1^2x_2= 4x_0^3-a_1x_0^2x_2-a_2x_2^3, \qquad 
2(x_0x_2-e_1x_2^2)\xi_2=x_1x_2\xi_1 \\ \\
x_1x_2\xi_2=2(x_0-e_2x_2)(x_0-e_3x_2)\xi_1 \qquad
\xi_3x_2=x_0\xi_1
\end{array}
$$

\appendix

\section{$\Pi$-Projective geometry revisited}\label{D-app}

We devote this appendix to reinterpret the $\Pi$-projective
line, discussed in Sec. \ref{pigeom-sec}, through the superalgebra $\bD$. 

\medskip

Let $\bD$ denote the \textit{super skew field}, $\bD=\C[\theta]$,
$\theta$ odd and $\theta^2=-1$. As a complex super vector space 
of dimension $1|1$, $\bD=\{a + b\theta\, |\, a,b \in \C\}$, thus it has a 
canonical structure of analytic supermanifold, and its functor of points is:
$$
T \mapsto \bD(T):=(\bD\otimes\cO(T))_0=\bD_0 \otimes \cO(T)_0 \oplus
\bD_1 \otimes \cO(T)_1
$$
Let $\bD^\times$ be the analytic supermanifold obtained by restricting the 
structure sheaf of the supermanifold $\bD$ to the open subset 
$\bD \setminus \{0\}$.

\smallskip\noindent
$\bD^\times$ is an analytic supergroup and its functor of points is:
$$
T \mapsto \bD^\times(T):=(\bD\otimes\cO(T))_0^*
$$
where $(\bD\otimes\cO(T))_0^*$ denotes the invertible elements in
$(\bD\otimes\cO(T))_0$;

As a supergroup $\bD^\times$ is isomorphic to $\bG_m^{1|1}$, which
is the supergroup with underlying topological space $\C^\times$, described
in Sec. \ref{pigeom-sec}. 
The isomorphism between $\bG_m^{1|1}$ and
$\bD^\times$ simply reads as:
$$
(a, \al) \mapsto a + \theta \al
$$
Notice that
$\bG_m^{1|1}$ (hence $\bD^\times$) is naturally embedded into $\rGL(1|1)$, 
the complex general linear supergroup
via the morphism (in the functor of points notation):
$$
\begin{array}{ccccc}
\bD^\times(T) & \cong & \bG_m^{1|1}(T) & \lra & \rGL(1|1)(T) \\ \\
 a + \theta \al & \cong & 
(a, \al) & \mapsto & \begin{pmatrix}a & \al \\ \al & a \end{pmatrix}
\end{array}
$$

\medskip\noindent
Before we continue this important characterization of $\Pi$-projective
geometry due to Deligne, let us point out that while $\bG_m^{1|1} \cong \bD^\times$
are commutative supergroups, the commutative algebra $\bD=\C[\theta]$
is not a commutative superalgebra, because if it were, then $\theta^2=0$
and not $\theta^2=-1$ as we have instead. This is an important fact, which
makes $\Pi$-projective supergeometry more similar to non commutative
geometry than to regular supergeometry.

\medskip

We now want to relate more closely the $\Pi$-projective supergeometry
with $\bD$.

\begin{lemma} \label{oddendo-lemma} 
A right action of $\bD$ on a complex super vector space $V$ is equivalent 
to the choice of an odd endomorphism $\phi$ of $V$ such that $\phi^2=1$.
\end{lemma}

\begin{proof}
Let $V$ be a right $\bD$-module. A right action of $\bD=\C[\theta]$ 
is an antihomomorphism $f: \bD \to \underline{\mathrm{End}}(V)$, 
which corresponds to a left action of the opposite algebra $\bD^o=\C[\theta^o]$,
$(\theta^o)^2=1$ ($\underline{\mathrm{End}}(V)$ denotes all of the endomorphisms
of $V$, not just the parity preserving ones).
Such actions are specified once we know the odd endomorphisms $\psi$ and $\phi$ corresponding respectively to $\theta$ and $\theta^o$.
Hence explicitly right multiplication 
by $\theta$ gives rise to an odd endomorphism $\phi$ such 
that $\phi^2 = 1$, by: 
\begin{align*}
\phi(v) := (-1)^{|v|} v \cdot \theta
\end{align*}
Conversely, given a super vector space $V$ and an odd endomorphism $\phi$
of square $1$, we can define a right $\bD$-module structure on $V$ by:
\begin{align*}
v \cdot (a + b \theta) := v \cdot a + (-1)^{|v|} \phi(v) \cdot b
\end{align*}
\end{proof}

Given any complex supermanifold $X$, there is a sheaf $\underline{\bD}$ of 
superalgebras, defined by $\underline{\bD}(U) := 
\mathcal{O}_X(U) \otimes_{\C} \bD$, for any open set $U \subseteq |X|$. 
Then a sheaf of right (resp. left) $\bD$-modules on $X$ is a sheaf of right 
(resp. left) modules for the sheaf $\underline{\bD}$; 
a {\it morphism} $\cF \to \cF'$ of sheaves of $\bD$-modules is 
simply a sheaf morphism that intertwines the $\bD$-actions on $\cF, \cF'$. 
A sheaf of $\bD$-modules $\mathcal{F}$ is {\it locally free of $\bD$-rank $n$} 
if $\mathcal{F}$ is locally isomorphic to $\underline{\bD}^n$.

\begin{lemma} \label{transfunctions-lemma}
Let $X$ be a complex supermanifold and
let $U$ be an open set. If $V$ is a free $\underline{\bD}(U)$-module 
of $\bD$-rank $1$, ${\mathrm{Aut}}_{\bD}(V) \cong 
\bD^\times(U)$.
\end{lemma}

\begin{proof}
Since $V$ is free of $\bD$-rank 1, we may reduce to the case $V = \underline{\bD}(U)$ as right $\bD$-modules, where this identification is obvious: $f \mapsto f(1) \in \underline{\bD}^\times(U)$ for $f \in \mathrm{Aut}_{\bD}(V)$. 
\end{proof}

\medskip

We are now ready to reinterpret the functor of points of the 
$\Pi$-projective line.

\begin{proposition}
Let the notation be as above.
\begin{align*}
P_\Pi(T) = \{\text{locally free, $\bD$-rank 1 right } 
\bD\text{-subsheaves } \mathcal{F_T} \subseteq \mathcal{O}_T^{2|2}\}
\end{align*}
In other words, the functor of points of the $\Pi$-projective line 
associates to each supermanifold $T$ the set of locally free
right $\bD$-subsheaves of rank $1|0$ of ${\cO}_T^{2|2}$. 
\end{proposition}

\begin{proof} (Sketch). The right action on $\cF_T$ of $\bD$ corresponds to the
right action of $\bD$ on $\C^{2|2}\otimes \cO_T $ occurring on the
first term through the left multiplication by the odd endomorphism $\phi$, 
(see Lemma \ref{oddendo-lemma}).
Hence the $\phi$-invariant subsheaves are in one to one correspondence
with the right $\bD$-subsheaves of  ${\cO}_T^{2|2}$.  Notice furthermore that by 
Lemma \ref{transfunctions-lemma}, the change of basis of the 
free module $\cF_T(V_i)$ we used in \ref{fopts-pi}, 
corresponds to right multiplication by an element of 
$\bG_m^{1|1} \cong \bD^\times$, that is
the natural left action of $\bD^\times$ on the locally free, $\bD$-rank $1$
sheaf  $\cF_T(V_i)$ by automorphisms.
\end{proof}

\begin{remark}
The generalization from $\C^\times$ action on $\C^n$ (see the construction
of ordinary projective space Sec. \ref{prelim-sec}) to
$\bG^{1|1}_m \cong \bD^\times$ action on $\C^{n|n}$ gives us naturally
the odd endomorphism $\phi$, which is used to construct the $\Pi$-projective
space and ultimately it is the base on which $\Pi$-projective geometry
is built. The introduction of the skew-field $\bD$, $\bD^\times \cong 
\bG^{1|1}_m$ is not merely a computational device, but suggest a more
fundamental way to think about $\Pi$-projective geometry. We are unable
to provide a complete treatment here, but we shall do so in a 
forthcoming paper.
\end{remark}

\end{document}